\documentclass[11pt]{article}
\usepackage{enumerate}
\usepackage{amssymb,a4wide,latexsym,makeidx,epsfig,fleqn}
\usepackage{amsthm}
\usepackage{amsmath}
\usepackage{enumerate}
\usepackage{extarrows}
\usepackage{graphicx}
\usepackage{subfigure}
\usepackage{float}
\usepackage[justification=centering]{caption}
\usepackage{tikz} 
\usepackage{amsmath}
\usepackage{caption}
\usepackage[numbers,sort&compress]{natbib}
\usetikzlibrary{positioning}
\usetikzlibrary{decorations.pathreplacing}
\newtheorem{theorem}{Theorem}[section]

\newtheorem{lemma}[theorem]{Lemma}
\newtheorem{fact}{Fact}

\begin{document}
\textwidth 150mm \textheight 225mm
\title{Spectral radius and size conditions for fractional $(a,b,m)$-covered graphs
\thanks{Supported by the National Natural Science Foundation of China (Nos. 12271439 and 12001434)}}
\author{Zengzhao Xu$^{a,b}$, Ligong Wang$^{a,b}$\footnote{Corresponding author.}, Weige Xi$^{c}$\\
{\small $^{a}$ School of Mathematics and Statistics, Northwestern Polytechnical University,}\\
{\small  Xi'an, Shaanxi 710129, P.R. China.}\\
{\small $^b$ Xi'an-Budapest Joint Research Center for Combinatorics, Northwestern
	Polytechnical University,}\\
{\small $^{c}$ College of Science, Northwest A\&F University, Yangling, Shaanxi 712100, China}\\
{\small E-mail: xuzz0130@163.com; lgwangmath@163.com; xiyanxwg@163.com}\\}
\date{}
\maketitle
\begin{center}
\begin{minipage}{120mm}
\vskip 0.3cm
\begin{center}
{\small {\bf Abstract}}
\end{center}
{\small  A fractional $(a,b,m)$-covered graph is a generalization of the concept of a fractional $[a,b]$-covered graph. For any $H \subseteq G$ with edge set $|E(H)| = m$, if there exists a fractional $[a,b]$-factor (the corresponding fractional indicator function is $h$) such that $h(e) = 1$ for any $e \in H$, then the graph $G$ is called a fractional $(a,b,m)$-covered graph. In this paper, we characterize the conditions for a graph to be a fractional $(a,b,m)$-covered graph from the perspectives of spectral radius and size, respectively. 

}
\end{minipage}
\end{center}

\section{Introduction }

 Graph-theoretic models can be used to efficiently model many real-world network systems. In an urban transportation network, different areas and the roads connecting them can be abstracted as vertices and edges in a graph, respectively.  Graphs and their factorization are widely used in network design, combinatorial design and other fields, thus attracting great academic attention. 
 
 Fractional factors play an important role in the fields of graph theory and combinatorial optimization. The fractional factor characterizes the existence of fractional flows in a graph, thereby representing the feasibility of data flow transmission in a network. The packet allocation problem in data transmission networks can be modeled as a fractional flow problem, which can be further transformed into a fractional factor existence problem. Therefore, the study of graph fractional factors is of great importance.
 
 Throughout this paper, we only consider finite, simple and undirected graphs. For a graph $G$, we use $V(G)$ and $E(G)$ to denote its vertex set and edge set, respectively. The number of edges in $G$ is denoted by $e(G)$. Let $ N_G(x) = \{y: xy \in E(G)\} $ and $ d(x) = |N(x)| $ denote the neighborhood and the degree of vertex $ x \in V(G) $, respectively. For any \( T \subseteq V(G) \), we denote \( G - T \) as the subgraphs induced by  $V(G) \setminus T$. For a subset \( M \subseteq E(G) \), the edge-deletion subgraph of \( G \) obtained by deleting all edges of \( M \) is denoted by \( G \setminus M \). Let \( N_T(x) = N_G(x)\cap T \) and \( d_T(x) = |N_T(x)| \). For two disjoint vertex subsets \( S_1, S_2 \subseteq V(G) \), let \( E_G(S_1, S_2) \) denote the set of edges in graph \( G \) with one endpoint in \( S_1 \) and the other in \( S_2 \), and let \( e_G(S_1, S_2) = |E_G(S_1, S_2)| \). For any $H \subseteq G$ with edge set $|E(H)| = m$, let $ e_H(S_1, S_2) = | \{ e = xy \in E(H) : x \in S_1, y \in S_2 \} |$. The minimum degree of $G$ is denoted by $\delta(G)$. Let $G_1 + G_2$ denote the disjoint union of graphs $G_1$ and $G_2$. The join of $G_1$ and $G_2$, denoted by $G_1 \lor G_2$, is derived from $G_1 + G_2$ by joining each vertex in $V(G_1)$ to each vertex of $V(G_2)$. For more details and concepts, readers may refer to \cite{BM}.

For a graph $G$, the adjacency matrix of $G$, denoted by $A(G) = (a_{ij})_{n \times n}$, is a matrix where $a_{ij} = 1 $ if vertex $v_i$ is adjacent to vertex $v_j$. Otherwise, $a_{ij}=0$. The largest eigenvalue
of $A(G)$ is called the spectral radius of $G$, denoted by $\rho(G)$.

For a graph $G$, consider two nonnegative integer-valued functions $g$ and $f$ defined on $V(G)$ with $g(u) \leq f(u)$ for all $u \in V(G)$. Let $ h : E(G) \to [0, 1] $ be a function defined on $E(G)$ of a graph $ G $. Let $E_G(u)$ denotes the set 
of edges incident with $u$ in $G$ and $E_h = \{e \in E(G) : h(e) > 0\}$. If $g(u) \leq \sum_{e\in E_G(u)} h(e) \leq f(u)$ holds for every 
$u \in V(G)$, then the subgraph of $G$ with $E_h$, denoted by $G[E_h]$, 
is called a fractional $(g,f)$-factor of $G$ with indicator function $h$. If $g(u) = a$ and $f(u) = b$ 
for every vertex $u \in V(G)$, then a fractional $(g,f)$-factor of $G$ is called a fractional $[a,b]$-factor of $G$. For more research on the fractional $(g, f)$-factor of graphs, readers may refer to \cite{FLLO,FLA,HLZZ,LFZ,LLGX,LMZ,LLA,MW,O,O2,WZ}).

The concept of a fractional covered graph, which dates back many years, requires that certain edges must be included in the corresponding fractional factor. Yang and Kang \cite{YK} first proposed the concept of a fractional $(g, f)$-covered graph, which was subsequently revised and refined by Li et al \cite{LYZ}. A graph $G$ is called a fractional $(g, f)$-covered graph if for every edge $e$ of $G$, there is a fractional $(g, f)$-factor with the indicator function $h$ such that $h(e) = 1$. If $g(u)=a$ and $f(u)=b$ with $1\le a\le b$ for each $u\in V(G)$, then a fractional $(g, f)$-covered graph is called a fractional $[a, b]$-covered graph. In particular, if $a=b=k$, then a fractional $[k,k]$-covered graph is called a fractional $k$-covered graph. In recent years, many scholars have begun to pay attention to fractional covered graphs. Zhou \cite{Z} gave a degree condition for a graph to be a fractional
$k$-covered graph. Yuan and Hao \cite{YH1} provided a degree condition that ensure a graph is a fractional $[a,b]$-covered graph. Later,  Yuan and Hao \cite{YH2} obtains some sufficient conditions based on the neighborhood union and binding number for a graph $G$ to be fractional $[a,b]$-covered. Wang et. al \cite{WZC} provided
a tight spectral radius condition for graphs to be fractional $[a, b]$-covered. However, the previous studies only discusses the case of covering a single edge.

\begin{theorem}(\cite{WZC})\label{T6}
	Let $1\le a \le b $ be integers, and let $ G $ be a graph of order $ n \geq 4 + \sqrt{32a^2 + 24a + 5} $. Let $H_{n,a}=K_{a-1}\lor(K_1+K_{n-a})$.
	
	\begin{enumerate}
		\item[(i)] For \( b \neq 1 \), if \( \rho(G) \geq \rho(H_{n,a}) \), then \( G \) is a fractional \([a, b]\)-covered graph unless \( G \cong H_{n,a} \).
		
		\item[(ii)] For \( b = 1 \), if \( \rho(G) \geq \rho(H_{n,3}) \), then \( G \) is a fractional \([1, 1]\)-covered graph unless \( G \cong H_{n,3} \).
	\end{enumerate}
\end{theorem}

For any $H \subseteq G$ with edge set $|E(H)| = m$, if there exists a fractional $[a,b]$-factor (the corresponding fractional indicator function is $h$) such that $h(e) = 1$ for any $e \in H$, then the graph $G$ is called a fractional $(a,b,m)$-covered graph. In particular, if $a=b=k$, then a fractional $(k,k,m)$-covered graph is called a fractional $(k,m)$-covered graph. Liu \cite{L} generalized the notion of fractional $(g, f )$-covered graphs by introducing the fractional $(g, f , m)$-covered graph. However, Gao and Wang \cite{GW} identified issues in Liu's work \cite{L} and they provided the revised necessary and sufficient condition that characterizes a fractional $(g, f, m)$-covered graph. Gao et. al \cite{GWC} presented the bounds of the isolated toughness and its variants for fractional $(a, b, m)$-covered graphs, and illustrates that the bounds were tight by counterexamples.

The spectral radius of a graph reflects many key properties of the graph. Thus, it is meaningful to study fractional covered graphs from the perspectives of spectral radius and size. In this paper, we characterize the conditions for a graph to be a fractional $(a,b,m)$-covered graph from the perspectives of spectral radius and size, respectively. Our main results are presented as follows.

\noindent\begin{theorem}\label{T1}  \  For integers $2\le a\le b$ and $1\le m\le b$, let $G$ be a connected graph of order $n\ge\frac{1}{2}[4b+2a+ab+(b+8)m+16]$ with minimum degree $\delta(G)\ge a+m$. If
	$$\rho(G)\ge n-b-1,$$
	then $G$ is a fractional $(a,b,m)$-covered graph. 
\end{theorem}

\noindent\begin{theorem}\label{T2}  \  
	\  For integers $2\le a\le b$ and $1\le m\le b$, let $G$ be a connected graph of order $n\ge 4a+\frac{5b}{2}+4m+7$ with minimum degree $\delta(G)\ge a+m$. If
	$$e(G)\ge \binom{n-b-1}{2}+ab+2a+(b+1)m,$$
	then $G$ is a fractional $(a,b,m)$-covered graph. 
\end{theorem}

If $a=b=k$, we can obtain conditions for a graph to be fractional $(k,m)$-covered.

\noindent\begin{theorem}  \  For integers $k\ge2$ and $1\le m\le k$, let $G$ be a connected graph of order $n\ge\frac{1}{2}[6k+k^2+(k+8)m+16]$ with minimum degree $\delta(G)\ge k+m$. If
	$$\rho(G)\ge n-k-1,$$
	then $G$ is a fractional $(k,m)$-covered graph. 
\end{theorem}

\noindent\begin{theorem}  \  
	\  For integers $k\ge2$ and $1\le m\le k$, let $G$ be a connected graph of order $n\ge \frac{13k}{2}+4m+7$ with minimum degree $\delta(G)\ge k+m$. If
	$$e(G)\ge \binom{n-k-1}{2}+k^2+2k+(k+1)m,$$
	then $G$ is a fractional $(k,m)$-covered graph. 
\end{theorem}

Particularly, if we set $m=1$ in Theorems \ref{T1} and \ref{T2}, then we can derive the spectral radius condition and size condition for a graph to be a fractional $[a,b]$-covered graph.

\noindent\begin{theorem}  \  For integers $2\le a\le b$, let $G$ be a connected graph of order $n\ge\frac{1}{2}[5b+2a+ab+24]$ with minimum degree $\delta(G)\ge a+1$. If
	$$\rho(G)\ge n-b-1,$$
	then $G$ is a fractional $[a,b]$-covered graph. 
\end{theorem}
\noindent\begin{theorem}  \  For integers $2\le a\le b$, let $G$ be a connected graph of order $n\ge 4a+\frac{5b}{2}+11$ with minimum degree $\delta(G)\ge a+1$. If
	$$e(G)\ge \binom{n-b-1}{2}+ab+2a+b+1,$$
	then $G$ is a fractional $[a,b]$-covered graph. 
\end{theorem}
By comparing with the results in \cite{WZC}, it is easy to see that $\rho(K_{a-1} \lor (K_{n-a} \cup K_1)) > \rho(K_{n-1}) = n-2 > n-b-1$ when $2\le a\le b $. Hence, our result makes a minor enhancement to the lower bounds of Theorem \ref{T6} \cite{WZC} through confining  $\delta \geq a+1$.

The remainder of this paper is organized as follows. In Section 2, we present several useful lemmas for proving the theorems in subsequent sections. In Section 3, we provide the proof of Theorem \ref{T1}. In Section 4, we provide the proof of Theorem \ref{T2}. 

\section{Preliminaries}

\quad\quad In this section, we present several useful lemmas for proving the theorems in subsequent sections. Recently, Gao and Wang \cite{GW} obtained the following necessary and sufficient condition for a graph to be a fractional 
$(g,f,m)$-covered graph.

\begin{lemma}(\cite{GW})\label{le:1}  \ 
		Let $G$ be a graph, $g$, $f$ be two non-negative integer-valued functions defined on $V(G)$ such that $g(x) \leq f(x)$ for each $x \in V(G)$. Let \( m \) be an integer with $0\le m \leq \min_{x \in V(G)} \{f(x)\}$. Then $G$ is a fractional $(g, f, m)$-covered graph if and only if for any $S \subseteq V(G)$,
		\[
		\sum_{x \in T} d_{G-S}(x)-\sum_{x \in T} g(x) +\sum_{x \in S} f(x) \geq \max_{\substack{H \subseteq G, |E(H)| = m}} \left\{\sum_{x \in S} d_H(x) - e_H(T, S) + \Theta(S, T)\right\},
		\]
		where $T = \{x : x \in V(G)\setminus S, d_{G - S}(x) \leq g(x) - 1\}$ and
		\[
		\Theta(S, T)=\sum_{\substack{1\leq d_{G\setminus E(H)-S}(x)-g(x)+d_H(x)\leq e_H(x,S)-1,e_H(x,S)\geq 2}}\{d_{G\setminus E(H)-S}(x)-g(x)+d_H(x)\}.
		\]
\end{lemma}

From Lemma \ref{le:1}, the following lemma is a natural consequence.

\begin{lemma}(\cite{GW,GWC})\label{le:2}
Let \( G \) be a graph, and let \( a, b, m \) be three integers with  \( 1\le a \leq b \) and \(0\le m \leq b \). Then \( G \) is a fractional \((a, b, m)\)-covered graph if and only if for any \( S \subseteq V(G) \), we have
	\[
	\sum_{v \in T} d_{G-S}(v) - a|T| + b|S| \geq \max_{\substack{H \subseteq G, |E(H)| = m}} \left\{ \sum_{x \in S} d_H(x) - e_H(T, S) + \Theta(S, T) \right\}, 
	\]
	where \( T = \bigl\{ x \mid x \in V(G)\setminus S,\ d_{G - S}(x) \leq a - 1 \bigr\} \) and
	\[
	\Theta(S, T) = \sum_{\substack{1 \leq d_{G \setminus E(H) - S}(x) - a + d_H(x) \leq e_H(x, S) - 1, e_H(x, S) \geq 2}} \bigl\{ d_{G \setminus E(H) - S}(x) - a + d_H(x) \bigr\}.
	\]
\end{lemma}

In the Remark 3 of \cite{GW}, based on the construction of the set $\Theta(S,T)$, Gao and Wang pointed out that the vertices in $\Theta(S,T)$ belong to $T$, and the following fact holds.

\begin{fact}(\cite{GW})
For any $H\subseteq G$ with $|E(H)|=m$, $\sum_{x \in S} d_H(x) - e_H(T, S) + \Theta(S, T)\le 2m$.
\end{fact}

In addition, based on Fact 1, Gao et al. \cite{GWC} provided a sufficient condition for a graph to be fractional $(a,b,m)$-covered.

\begin{lemma}(\cite{GWC})
	Let \( G \) be a graph, and let \( a, b, m \) be three integers with  \( 1\le a \leq b \) and \(0\le m \leq b \). For any \( S \subseteq V(G) \), if
	
	\[
	\sum_{v \in T} d_{G-S}(v) - a|T| + b|S| \geq 2m,
	\]
where \( T = \bigl\{ x \mid x \in V(G) \setminus S,\ d_{G - S}(x) \leq a - 1 \bigr\} \), then \( G \) is fractional \((a, b, m)\)-covered.
\end{lemma}

Using the fundamentals of non-negative matrices, we present a classical result on the comparison of the spectral radius between a graph and its subgraph.

\begin{lemma}(\cite{BA})\label{le:3} \ Let $ G $ be a connected graph and $G^{*} $ be a subgraph of $ G $. Then 
	$$ \rho(G^{*}) \leq \rho(G), $$
	with the equality holds if and only if $ G^{*} \cong G $.
\end{lemma}

Finally, we present a classical result concerning upper bounds for the spectral radius.

\begin{lemma}(\cite{HSF,N})\label{le:4} \ Let $G$ be a graph on $n$ vertices and $m$ edges with minimum degree $\delta \geq 1$. Then
	$$
	\rho(G) \leq \frac{\delta - 1}{2} + \sqrt{2m - n\delta + \frac{(\delta + 1)^2}{4}},
	$$
	with equality if and only if $G$ is either a $\delta$-regular graph or a bidegreed graph in which each vertex is of degree either $\delta$ or $n - 1$.
\end{lemma}

\begin{lemma}(\cite{HSF,N})\label{le:5} \ For nonnegative integers $p$ and $q$ with $2q \leq p(p - 1)$ and $0 \leq x \leq p - 1$, the function
	$$
	f(x) = \frac{x - 1}{2} + \sqrt{2q - px + \frac{(1 + x)^2}{4}}
	$$ 
	is decreasing with respect to $x$.
\end{lemma}

\section{Proof of Theorem \ref{T1}}

In this section, we provide the proof of Theorem \ref {T1}, which devises a spectral radius condition to ascertain a graph is a fractional $(a,b,m)$-covered graph. 

\begin{proof}[Proof of Theorem~\ref{T1}]  %

By contradiction, suppose that $G$ is not fractional $(a,b,m)$-covered, where $2\le a\le b$ and $1\le m\le b$. By Lemma \ref{le:2}, there exists $S \subseteq V(G)$ and $H\subseteq G$ with $|E(H)|=m$ such that
$ \sum_{v \in T} d_{G-S}(v)-a|T|+b|S| \le \Delta(S,T,H)-1 $, where $T = \{x \in V(G) \setminus S \mid d_{G-S}(x) \leq a-1\}$, $\Delta(S,T,H)=\sum_{x \in S} d_H(x) - e_H(T, S) + \Theta(S, T)$ and $$\Theta(S, T) = \sum_{\substack{1 \leq d_{G \setminus E(H) - S}(x) - a + d_H(x) \leq e_H(x, S) - 1, \\ e_H(x, S) \geq 2}} \bigl\{ d_{G \setminus E(H) - S}(x) - a + d_H(x) \bigr\}.
$$ Let $t=|T|$ and $s=|S|$, by Fact 1, we have 

\begin{equation}\label{eq:1}
	\sum_{v \in T} d_{G-S}(v)-at+bs \le \Delta(S,T,H)-1\le 2m-1. 
\end{equation}

Next, we will prove three claims.

{\bf Claim 1.} $s\ge m+1$.

{\bf Proof.} We first prove that $T\neq \varnothing$. Otherwise, $T= \varnothing$. Then we have $e_H(T, S)=0$ and $\Theta(S,T)=0$. Note that $m\le b$, we have $0=\sum_{v \in T} d_{G-S}(v)\le 2m-1-bs+a \cdot 0$. Then we obtain $bs\le 2m-1\le 2b-1$. Hence $s\le 1$.

 If $s=0$, we have $\Delta(S,T,H)=\sum_{x \in S} d_H(x) - e_H(T, S) + \Theta(S, T)=0.$ According to (\ref{eq:1}), we have $0=\sum_{v \in T} d_{G-S}(v)-at+bs \le \Delta(S,T,H)-1=-1$, which is a contradiction.
 
 If $s=1$, then $\sum_{v \in T} d_{G-S}(v)-at+bs=b.$ Let $S=\{v_0\}$, we obtain $$\Delta(S,T,H)=\sum_{x \in S} d_H(x) - e_H(T, S) + \Theta(S, T)=d_H(v_0)\le m .$$ However, by (\ref{eq:1}), we have $$\Delta(S,T,H)-1\ge \sum_{v \in T} d_{G-S}(v)-at+bs=b\ge m.$$ Hence $\Delta(S,T,H)\ge m+1$, which is a contradiction. Hence $T\neq \varnothing$. 
 
 We proceed to prove $s\ge m+1$. For any $u\in T$, we have $d_G(u)=d_{G-S}(u)+e(u,S).$ Since $\delta(G)\ge a+m$ and $d_{G-S}(u)\le a-1$ for any $u\in T$, we obtain
 \begin{align*}
 	e(u,S)&=d_G(u)-d_{G-S}(u)\\
 	&\ge \delta(G)-d_{G-S}(u)\\
 	&\ge (a+m)-(a-1)\\
 	&=m+1.
 \end{align*}
Hence $s\ge m+1$. $\quad\Box$

{\bf Claim 2.} $t\ge b+1.$

{\bf Proof.} Recall that $1\le m\le b$, $s\ge m+1$, $\delta(G) \geq a+m$ and $d_{G-S}(v) \geq \delta(G) - s \geq a+m - s$ for each $v \in T$. According to (\ref{eq:1}), we have
$$
(a+m-s)t \leq \sum_{v \in T} d_{G-S}(v) \leq 2m-1+at-bs.
$$
Therefore $t \geq \frac{bs-2m}{s-m} + \frac{1}{s-m}\geq \frac{bs-bm}{s-m} + \frac{1}{s-m}= b + \frac{1}{s-m}$. Since $t$ is a positive integer and $s \geq m+1$ due to Claim 1, then $t \geq b + 1$.$\quad\Box$

{\bf Claim 3.} $s\le t+m-1$.

{\bf Proof.}  Otherwise, $s \geq t+m$. By \eqref{eq:1} and $2\le a\le b$,
$$
0 \leq \sum_{v \in T} d_{G-S}(v) \leq 2m-1+at- bs \leq 2m-1+at- b(t+m)=(2-b)m+(a-b)t-1\le -1,
$$
a contradiction. Thus, $s \leq t+m-1$.$\quad\Box$

By (\ref{eq:1}), we get
\begin{align*}
	e(G) &=e(G-S-T, T)+e(T) +e(S, T) + e(G - T)\\
	&\leq \sum_{v \in T} d_{G-S}(v) + e(S, T) + e(G - T) \\
	&\leq 2m-1+at- bs + st + \binom{n - t}{2} \\
	&= 2m-1+at- bs + st + \frac{(n - t)(n - t - 1)}{2}.
\end{align*}

By integrating the condition $\delta(G) \geq a+m$ with Lemmas \ref{le:4} and \ref{le:5}, we obtain
{\small\begin{align*}
	\rho(G) &\leq \frac{\delta(G) - 1}{2} + \sqrt{2e(G) - n\delta(G) + \frac{(\delta(G) + 1)^2}{4}} \\
	&\leq \frac{a+m - 1}{2} + \sqrt{2e(G) - n(a+m) + \frac{(a+m + 1)^2}{4}} \\
	&\leq \frac{a+m-1}{2} + \sqrt{2\left(2m-1+at- bs + st + \frac{(n - t)(n - t - 1)}{2}\right) - n(a+m) + \frac{(a+m+ 1)^2}{4}} \\
	&= \frac{a+m- 1}{2}+\sqrt{\left(n - b - \frac{a+m+ 1}{2}\right)^2-f(t)},
\end{align*}
}
where $f(t)=-t^2 + (2n- 2a -2s- 1)t+ b^2+ab+(b-4)m+2bs-2bn+b+2$. Recall Claims 2 and 3, we obtain $t\ge b+1$ and $s\le t+m-1$. Since  $\frac{\partial f}{\partial s}=-2t+2b<0$, we obtain
\begin{align*}
	f(t) &=-t^2 + (2n- 2a -2s- 1)t+ b^2+ab+(b-4)m+2bs-2bn+b+2 \\
	&\geq-t^2 + (2n- 2a -2(t+m-1)- 1)t+ b^2+ab+(b-4)m+2b(t+m-1)-2bn+b+2 \\
	&= - 3 t^2+ (2n+ 2 b-2a - 2m + 1) t+ b^2+ab+(3b-4)m-2bn-b+2.
\end{align*}
 
Let $g(t)= - 3 t^2+ (2n+ 2 b-2a - 2m + 1) t+ b^2+ab+(3b-4)m-2bn-b+2.$ Since  $n\ge\frac{1}{2}[4b+2a+ab+(b+8)m+16]$, then the symmetry axis of $g(t)$ is $\bar{t}=\frac{2n+2b-2a-2m+1}{6}>b+1.$ Note that $t\ge b+1$, we distinguish the following two cases based on the value of $t$.

{\bf Case 1.} $b+1\le t\le\frac{n}{2}.$

Since $n\ge\frac{1}{2}[4b+2a+ab+(b+8)m+16]$, we have
\begin{align*}
	g(b+1) &=2n-4b-2a-ab-(6-b)m \\
	&\geq 2(\frac{1}{2}[4b+2a+ab+(b+8)m+16])-4b-2a-ab-(6-b)m\\
	&\ge 2(b+1)m+16\\
	&>0.
\end{align*}
and 
\begin{align*}
	g(\frac{n}{2}) &= \frac{n^2}{4}-(b+a+m-\frac{1}{2})n+ b^2+ab+(3b-4)m -b+2 \\
	&\geq a^2(\frac{b^2}{16}-\frac{b}{4}-\frac{3}{4})+a(\frac{b^2m}{8}+\frac{bm}{4}+\frac{5b}{4}-3m-\frac{7}{2})+6m+\frac{13bm}{4}+\frac{bm^2}{2}+\frac{b^2m^2}{16}+22\\
	&\ge a^2(\frac{b^2}{16}-\frac{b}{4}-\frac{3}{4})+a(\frac{b^2m}{8}+\frac{bm}{4}+\frac{5b}{4}-3m-\frac{7}{2})+6m+22.
\end{align*}

Suppose $y(a,b)=a^2(\frac{b^2}{16}-\frac{b}{4}-\frac{3}{4})+a(\frac{b^2m}{8}+\frac{bm}{4}+\frac{5b}{4}-3m-\frac{7}{2})+6m+22$. To show that $g(\frac{n}{2})>0$, we will prove $y(a,b)>0$.

Since  $2\le a\le b$, we obtin $\frac{\partial y}{\partial b}=a^2(\frac{b}{8}-\frac{1}{4})+a(\frac{bm}{4}+\frac{m}{4}+\frac{5}{4})>0$.  Then
$y(a,b)\ge y(a,2)=y(2,2)=2m+16>0\quad\text{(if $b=2$, then $a=2$).}$
Thus, we have $g(\frac{n}{2})\ge y(a,b)>0.$ Based on the above procedure, we have $g(t)\ge \min\{g(b+1),g(\frac{n}{2})\}>0$ for $b+1\le t\le\frac{n}{2}$. Hence $f(t)\ge g(t)>0$ for $b+1\le t\le\frac{n}{2}$. 

{\bf Case 2.} $t\ge \frac{n+1}{2}.$

 Since $n\ge s+t$,  we obtain $s\le n-t$ and
\begin{align*}
	f(t) & =-t^2 + (2n- 2a -2s- 1)t+ b^2+ab+(b-4)m+2bs-2bn+b+2 \\
	&\geq t^2 - ( 2b+2a + 1)t +b^2+ ab+(b-4)m+b+2 
	\quad \text{(since $s \leq n - t$ and $t \geq b + 1$)} \\
	&\geq \frac{n^2}{4} - (b+a)n + b^2+ ab+(b-4)m-a+ \frac{7}{4}
	\quad \text{(since $t \geq \frac{n+1}{2}$ and $n\ge\frac{1}{2}[4b+2a+ab$} \\
	& +(b+8)m+16]\\
	&\geq a^2(\frac{b^2}{16}-\frac{b}{4}-\frac{3}{4})+a(\frac{b^2m}{8}+\frac{3bm}{4}+b-2m-5)+12m+3bm+(4+b)m^2+\frac{b^2m^2}{16}+ \frac{71}{4}\\
	&\ge a^2(\frac{b^2}{16}-\frac{b}{4}-\frac{3}{4})+a(\frac{b^2m}{8}+\frac{3bm}{4}+b-2m-5)+12m+\frac{71}{4}.
\end{align*}
Let $q(a,b)=a^2(\frac{b^2}{16}-\frac{b}{4}-\frac{3}{4})+a(\frac{b^2m}{8}+\frac{3bm}{4}+b-2m-5)+12m+\frac{71}{4}$. Since $2\le a\le b$, we obtin  $\frac{\partial q}{\partial b}=a^2(\frac{b}{8}-\frac{1}{4})+a(\frac{bm}{4}+\frac{3m}{4}+1)>0.$ Therefore,  
$$q(a,b)\ge q(a,2)=q(2,2)=12m+\frac{31}{4}>0\quad\text{(if $b=2$, then $a=2$)}.$$
Then $f(t)\ge q(a,b)>0$ when $t\ge\frac{n+1}{2}$. 

Hence $f(t)>0$ for $t\ge b+1$, which implies that $$\rho(G) \leq  \frac{a+m - 1}{2} +\sqrt{\left(n - b - \frac{a+m+ 1}{2}\right)^2 - f(t)}<n-b-1.$$ This contradicts the statement that $\rho(G) \ge  n-b-1.$

This completes the proof.
\end{proof}

\section{Proof of Theorem \ref{T2}}

In this section, we provide the proof of Theorem \ref {T2}, which devises a size condition to ascertain a graph is a fractional $(a,b,m)$-covered graph. 

\begin{proof}[Proof of Theorem~\ref{T2}]  %
	
	By contradiction, suppose that $G$ is not fractional $(a,b,m)$-covered, where $2\le a\le b$ and $1\le m\le b$. By Lemma \ref{le:2}, there exists $S \subseteq V(G)$ and $H\subseteq G$ with $|E(H)|=m$  such that
	$ \sum_{v \in T} d_{G-S}(v)-a|T|+b|S| \le \Delta(S,T,H)-1 $, where $T = \{x \in V(G)\setminus S \mid d_{G-S}(x) \leq a-1\}$, $\Delta(S,T,H)=\sum_{x \in S} d_H(x) - e_H(T, S) + \Theta(S, T)$ and $$\Theta(S, T) = \sum_{\substack{1 \leq d_{G \setminus E(H) - S}(x) - a + d_H(x) \leq e_H(x, S) - 1, \\ e_H(x, S) \geq 2}} \bigl\{ d_{G \setminus E(H) - S}(x) - a + d_H(x) \bigr\}.	$$ 
	
	Let $t=|T|$ and $s=|S|$, by Fact 1, we have 
	
	\begin{equation}\label{eq:2}
		\sum_{v \in T} d_{G-S}(v)-at+bs \le \Delta(S,T,H)-1\le 2m-1. 
	\end{equation}
	
	Based on the proof of Theorem \ref{T1}, we can obtain $s\ge m+1$ and $t\ge b+1$.
	
	By (\ref{eq:2}), we obtain
	\begin{equation}\label{eq:3}
		\begin{aligned}
			e(G) &=e(G-S-T, T)+e(T) +e(S, T) + e(G - T)\\
			&\leq \sum_{v \in T} d_{G-S}(v) + e(S, T) + e(G - T) \\
			&\leq 2m-1+at- bs  + st + \binom{n - t}{2} \\
			&= 2m-1+at- bs  + st + \frac{(n - t)(n - t - 1)}{2}\\
			&=\binom{n-b-1}{2}+ab+2a+(b+1)m-f(t),
		\end{aligned}
	\end{equation}
where $f(t) = -\frac{t^2}{2} + \left(n - a-s - \frac{1}{2}\right)t + \frac{b^2}{2}+ab+ \frac{3}{2}b+(b-1)m-(b+1)n+bs+2a+2.$
	
	Next, we will prove that $f(t)>0$ based on the value of $t$.
	
	{\bf Case 1.} $b+1\le t\le \frac{n}{2}.$
	
	{\bf Case 1.1.} $a=b=k$.

Then $f(t) = -\frac{t^2}{2} + \left(n - k-s - \frac{1}{2}\right)t + \frac{3k^2}{2}+ \frac{7}{2}k+(k-1)m-(k+1)n+ks+2.$ Then we prove that $t\ge s-m+1$. Otherwise, we obtain $t\le s-m.$ Since $k\ge 2$, by (\ref{eq:2}), 
	$$	0 \leq \sum_{v \in T} d_{G-S}(v) \leq 2m-1+kt-ks  \leq 2m-1+k(s-m) - ks =(2-k)m -1\le -1,	$$
	a contradiction. Then $s \leq t + m - 1$. Since $t\ge k+1$ and $\frac{\partial f}{\partial s}=-t+k<0$, we have  
{\small	\begin{align*}
		f(t) &= -\frac{t^2}{2} + \left(n - k-s - \frac{1}{2}\right)t + \frac{3k^2}{2}+ \frac{7}{2}k+(k-1)m-(k+1)n+ks+2 \\
		&\geq -\frac{t^2}{2} + \left(n - k-(t + m - 1) - \frac{1}{2}\right)t + \frac{3k^2}{2}+ \frac{7}{2}k+(k-1)m-(k+1)n+k(t + m - 1)+2 \\
		&= -\frac{3}{2}t^2 + \left(n - m+ \frac{1}{2}\right)t + \frac{3}{2}k^2+ \frac{5}{2}k+(2k-1)m-(k+1)n + 2.
	\end{align*}}
	
	Let $q_1(t)=-\frac{3}{2}t^2 + \left(n - m+ \frac{1}{2}\right)t + \frac{3}{2}k^2+ \frac{5}{2}k+(2k-1)m-(k+1)n + 2.$ Since $n\ge \frac{13k}{2}+4m+7$, $k\ge2$, by a simple calculation, we have
	$$q_1(k+1)=(k-2)m+1>0$$ and
	\begin{align*}
		q_1\left(\frac{n}{2}\right) &= \frac{n^2}{8} - \left( \frac{m}{2} +k+ \frac{3}{4}\right)n  + \frac{3}{2}k^2 +\frac{5}{2}k+(2k-1)m + 2 \\
		& \ge \frac{9k^2}{32}-\frac{m}{2}+\frac{5km}{4}+2k+\frac{23}{8} \quad \text{(since $n\ge \frac{13k}{2}+4m+7$)}\\
		& > \frac{9k^2}{32}+\frac{23}{8}\\
		& >0.
	\end{align*}
Hence, if $k + 1 \leq t \leq \frac{n}{2}$, we have
	$$
	q_1(t) \geq \min\left\{q_1(k + 1), q_1\left(\frac{n}{2}\right)\right\} > 0.
	$$
	It follows that $f(t) \geq q_1(t) > 0$ for $k + 1 \leq t \leq \frac{n}{2}$.

	{\bf Case 1.2.} $a<b$.
	
	{\bf Subcase 1.2.1.} $s > a - \frac{1}{b-a}+m$.
	
	Then we obtain $t \geq s + b-a -m+1$. Otherwise, $t \leq s + b-a-m$. By  (\ref{eq:2}), $2\le a<b$ and $s > a - \frac{1}{b-a}+m$, 
	$$
	0 \leq \sum_{v \in T} d_{G-S}(v) \leq 2m-1+at- bs  <2m-1+(b-a)(a-s)-am<(2-b)m<0,
	$$
	a contradiction. Therefore $s \leq t -(b-a) +m - 1$. Since $t\ge b+1$ and $\frac{\partial f}{\partial s}=-t+b<0,$ we have
	\begin{align*}
		f(t)& = -\frac{t^2}{2} + \left(n - a-s - \frac{1}{2}\right)t + \frac{b^2}{2}+ab+ \frac{3}{2}b+(b-1)m-(b+1)n+bs+2a+2\\
		 &\geq -\frac{3t^2}{2} + \left(n + 2(b-a)-m + \frac{1}{2}\right)t- \frac{b^2}{2}+2ab+(2b-1)m-(b+1)n +\frac{b}{2}  + 2a + 2.
	\end{align*}
	Let  
	$$
	q_2(t)=-\frac{3t^2}{2} + \left(n + 2(b-a)-m + \frac{1}{2}\right)t- \frac{b^2}{2}+2ab+(2b-1)m-(b+1)n +\frac{b}{2}  + 2a + 2.
	$$
	Since $2\le a< b$ and $n\ge 4a+\frac{5b}{2}+4m+7$, we have
	$$
	q_2(b + 1) = (b-2)m+1 > 0
	$$
	and
	\begin{align*}
		q_2\left(\frac{n}{2}\right) &= \frac{n^2}{8} - \left(a + \frac{m}{2}+\frac{3}{4}\right)n - \frac{b^2}{2}+ 2ab+\frac{b}{2}  +(2b-1)m + 2a + 2 \\
		&\geq  \frac{9b^2}{32}+(\frac{13}{4}b-2a-\frac{1}{2})m+2a(b-a)+3b-a + \frac{23}{8} \quad \text{(since $n\ge 4a+\frac{5b}{2}+4m+7$)}\\
		&>  \frac{9b^2}{32} + \frac{23}{8}\\
		&>0
	\end{align*}
Therefore,
	$$
	q_2(t) \geq \min\left\{q_2(b + 1), q_2\left(\frac{n}{2}\right)\right\} > 0.
	$$	for $b + 1 \leq t \leq \frac{n}{2}$. It follows that $f(t) \geq q_2(t) > 0$ for $b + 1 \leq t \leq \frac{n}{2}$.
	
	{\bf Subcase 1.2.2.} $s \le a - \frac{1}{b-a}+m$.
	
	Since $s \leq a - \frac{1}{b-a}+m$, we obtain
{\small	\begin{align*}
		f(t)& = -\frac{t^2}{2} + \left(n - a-s - \frac{1}{2}\right)t + \frac{b^2}{2}+ab+ \frac{3}{2}b+(b-1)m-(b+1)n+bs+2a+2\\
		 &\geq -\frac{t^2}{2} + \left(n+\frac{1}{b-a} - 2a-m - \frac{1}{2}\right)t 
		+ \frac{b^2}{2}+2ab+ \frac{3}{2}b -\frac{b}{b-a}+(2b-1)m-(b+1)n  + 2a + 2.
	\end{align*}}
   Let 
	$$
	q_3(t) = -\frac{t^2}{2} + \left(n+\frac{1}{b-a} - 2a-m - \frac{1}{2}\right)t 
	+ \frac{b^2}{2}+2ab+ \frac{3}{2}b -\frac{b}{b-a}+(2b-1)m-(b+1)n  + 2a + 2.
	$$
Since $2\le a< b$ and $n\ge 4a+\frac{5b}{2}+4m+7$, we have
	$$
	q_3(b + 1) =  \frac{1}{b-a}+(b-2)m+1 > 0
	$$
	and
	\begin{align*}
		q_3\left(\frac{n}{2}\right) &= \frac{3n^2}{8} - \left(a + b+\frac{1}{2}m+ \frac{5}{4} - \frac{1}{2(b-a)} \right)n + \frac{b^2}{2}+ 2ab + \frac{3}{2}b+(2b-1)m- \frac{b}{b-a}+ 2a + 2  \\
		&\geq \frac{11b^2}{32}+\frac{8m+b+8a+14}{4(b-a)}+4m^2+(6a+\frac{17}{4}b+\frac{23}{2})m+2a^2+3ab+\frac{9}{2}b+11a+\frac{93}{8}\\
		&\geq \frac{11b^2}{32}+\frac{93}{8}\\
		&> 0
	\end{align*}
Hence, for $b + 1 \leq t \leq \frac{n}{2}$, we get
	$$
	q_3(t) \geq \min\left\{q_3(b + 1), q_3\left(\frac{n}{2}\right)\right\} > 0.
	$$
It follows that  $f(t) >q_3(t)>0$ in this case. Hence $f(t)>0$ for $b + 1 \leq t \leq \frac{n}{2}$.

 Combining this with (\ref*{eq:3}), we obtain $e(G) < \binom{n - b - 1}{2} + ab + 2a+(b+1)m$, a contradiction.
	
	{\bf Case 2.} $t\ge \frac{n+1}{2}.$
	
Note that $n \geq s + t$,  then $s \leq n - t$. For $t \geq \frac{n+1}{2}$, we obtain
	\begin{align*}
		f(t) & = -\frac{t^2}{2} + \left(n - a-s - \frac{1}{2}\right)t + \frac{b^2}{2}+ab+ \frac{3}{2}b+(b-1)m-(b+1)n+bs+2a+2\\
		&\geq \frac{t^2}{2} - \left(a +b + \frac{1}{2}\right)t+ \frac{b^2}{2}+ ab+ \frac{3b}{2} +(b-1)m-n + 2a + 2 \quad \text{(since $s \leq n - t$)} \\
		&\geq \frac{n^2}{8} -(\frac{a+b}{2}+1)n+\frac{b^2}{2}+ ab+ \frac{3a}{2}+(b-1)m+b  + \frac{15}{8}  \quad \text{(since $t \geq \frac{n+1}{2}$)} \\
		&\geq \frac{b^2}{32} +\frac{b}{4}(a-\frac{5}{2})+2(a+1)m+\frac{3bm}{2}+2m^2+a+1 \quad \text{(since $n\ge 4a+\frac{5b}{2}+4m+7$)}\\
		&\ge\frac{b^2}{32} +\frac{b}{4}(a-\frac{5}{2})+a+1 \quad \text{(since  $2\le a\le b$ and $k\ge0$)} 
	\end{align*}
	
	Let $w(a,b)=\frac{b^2}{32} +\frac{b}{4}(a-\frac{5}{2})+a+1$. Then $\frac{\partial w}{\partial b}=\frac{1}{16}b+\frac{1}{4}(a-\frac{5}{2})\ge\frac{1}{16}b-\frac{1}{8}>0$ due to $2\le a\le b$. Hence $w(a,b)\ge w(a,2)=w(2,2)=\frac{23}{8}>0$, we have $f(t)>0$ for $t\ge \frac{n+1}{2}$. By (\ref{eq:3}), we get $e(G) < \binom{n-b-1}{2} + ab + 2a+(b+1)m$, which also leads to a contradiction.
	
	This completes the proof.
\end{proof}

\section*{Declarations}

The authors declare that they have no conflict of interest.

\section*{Data availability}

No data was used for the research described in the paper.











\end{document}